\documentclass[letterpaper,boxed,11pt]{article}
\usepackage[margin=1in]{geometry}
\usepackage{amsmath,pgfplots}
\usepackage{enumerate}
\usepackage{amssymb,amsthm}
\usepackage{hyperref}
\usepackage{subcaption}
\usepackage{graphicx}

\theoremstyle{plain}
\newtheorem{theorem}{Theorem}[section] 

\theoremstyle{definition}

\newtheorem{cor}[theorem]{Corollary}
\newtheorem{lemma}[theorem]{Lemma}
\newtheorem{prop}[theorem]{Proposition}
\newtheorem{rem}[theorem]{Remark}

\newcommand{\BHV}{\mbox{BHV}}
\newcommand{\Isom}{\mbox{Isom}}
\newcommand{\trop}{M_{0,n}^{\mathrm{trop}}}

\title{The isometry group of phylogenetic tree space is $S_n$}
\author{Gillian Grindstaff}
\date{\today}

\begin{document}
\maketitle
\begin{abstract}
A phylogenetic tree is an acyclic graph with distinctly labeled leaves, whose internal edges have a positive weight. Given a set $\{1,2,\dots,n\}$ of $n$ leaves, the collection of all phylogenetic trees with this leaf set can be assembled into a metric cube complex known as phylogenetic tree space, or Billera-Holmes-Vogtmann tree space, after \cite{BHV}. In this largely combinatorial paper, we show that the isometry group of this space is the symmetric group $S_n$. This fact is relevant to the analysis of some statistical tests of phylogenetic trees, such as those introduced in \cite{BBW}.
\end{abstract}

\section{Introduction}
 Phylogenetic trees commonly represent evolutionary or branching relationships between a set of organisms or samples, which we represent as the label set $\{1,2,\dots,n\}$, quantified by degree of genotypic or phenotypic distance (see Figure \ref{fig:1a}).

 As defined by Billera, Holmes, and Vogtmann in \cite{BHV}, phylogenetic tree space $\BHV_n$, for leaves labeled $\{1,2,\dots, n\}$, consists of $(2n-5)!!$ unbounded cubes corresponding to unique binary tree shapes (``topologies") up to label-preserving graph isomorphism. Each cube can be parametrized by the $n-3$ edge lengths in that topology, which gives a homeomorphism to $\mathbb{R}^{n-3}$. This is called an orthant.\footnote{WARNING! In the original paper, and indeed in a majority of the literature, the leaves of $\BHV_n$ are labeled $\{0,1,\dots, n-1,n\}$, and therefore tree space has $(2n-3)!!$ orthants of dimension $n-2$. Our notation (labels starting at 1, $BHV_n$ having $n$-leafed trees) simplifies the proofs, emphasizes unrootedness, and aligns more intuitively with the symmetries, but this change of notation must be kept in mind; in particular, the isometry group of $\BHV_n$ is isomorphic to $S_{n+1}$ in the standard notation.}
 
\begin{figure}[h]
\label{sector}

\begin{subfigure}{0.48\textwidth}\begin{center}
\begin{tikzpicture}[scale=.6,auto=left,every node/.style={circle,fill=none}]
  \tikzset{edge/.style = {->,> = latex'}}
  \node (n0) at (5.5,7) {6};
  \node (n1) at (1,0)  {1};
  \node (n3) at (2.5,0) {2};
  \node (n4) at (4,0)  {3};
  \node (n5) at (5.5,0)  {4};
  \node (n6) at (7,0) {5};
  
  \node (n25) at (6.5,5) {0.25};
  \node (n30) at (5,4) {0.3};
  \node (n45) at (7,3) {0.45};

  \node[fill=none,circle,inner sep=0pt,minimum size=1pt] (n12) at (6,4.5) {};
  \node[fill=none,circle,inner sep=0pt,minimum size=1pt] (n14) at (4.5,2.5)  {};
  \node[fill=none,circle,inner sep=0pt,minimum size=1pt] (n15) at (6.5,1.5) {};
  \node[fill=none,circle,inner sep=0pt,minimum size=1pt] (n19) at (5.5,6)  {};
  \foreach \from/\to in {n1/n19,n0/n19,n3/n14,n4/n14,n5/n15,n6/n15}
    { \draw (\from) -> (\to); \draw (\from) -> (\to);  }
    
  \foreach \from/\to in {n12/n19,n14/n12,n15/n12}
    {  \draw[thick,red] (\from) to (\to); }
\end{tikzpicture}
\caption{A tree $T\in \BHV_6$. $T$ has 6 external (``leaf") edges, and 3 internal edges with weights as labeled.}\label{fig:1a}
 \end{center}
\end{subfigure}
\hspace{1cm}
\begin{subfigure}{0.48\textwidth}
\begin{center}
\begin{tikzpicture}
  [scale=.9,auto=left]
  \node (n0) at (5,5) {(16)(2345)};
  \node (n1) at (1,1)  {c};
  \node (n2) at (1,3)  {0.3};
  \node (n3) at (4,1) {0.45};
  \node (n4) at (2,2)  {0.25};
  \node (n5) at (5,2)  {};
  \node (n6) at (5,4) {T};
  \node (n7) at (1,5)  {(23)(1456)};
  \node (n8) at (6,1)  {(45)(1236)};
  \node (n9) at (2,4) {};
  
  \foreach \from/\to in {n1/n2,n1/n3,n1/n4}
    {  \draw[thick,red] (\from) to (\to); }
  \foreach \from/\to in {n3/n5,n5/n6,n4/n5,n4/n9,n9/n6,n2/n9}
    {  \draw[dotted] (\from) to (\to); }
  \foreach \from/\to in {n2/n7,n3/n8,n4/n0}
    {  \draw[very thin,->] (\from) to (\to); }
\end{tikzpicture}
\caption{The orthant of $\BHV_6$ [$\cong(\mathbb{R}^3)^{\geq 0}$] containing $T$, with an axis for each unweighted edge (``partition") of $T$. The axes are parametrized by edge length, so the point $T$ is graphed above in relation to the other trees of identical topology. }\label{fig:1b}
\end{center}
\end{subfigure}
\end{figure} 
On an orthant boundary component of codimension $k$, $k$ edges have length 0, which leaves $n-k-3$ internal edges which may be present in a number of other binary orthants. This number is bounded in Lemma \ref{volumebound}, which may be of independent interest. We then identify the orthant boundaries according to this (weighted labeled graph) equivalence. In particular, at the ``origin", every orthant contains the star-shaped tree having no internal edges of positive length. Once identified, this is called the \textit{cone point }$c$ (see Figure \ref{fig:1b}), well-named because for a particular simplicial complex $L_n$, it is the image of the quotient $\BHV_n = L_n \times [0,\infty)/(L_n \times 0)$. \cite{BHV}

A metric on $\BHV_n$ is generated by the Euclidean metric within each orthant: a path $\gamma$ between trees $T$ and $T'$ has length 
$$ \ell(\gamma) = \sum_{S \in \mathcal{O}} |\gamma \cap S|,$$
where $|\cdot|$ is Euclidean path length via restriction to an orthant. Then 
$$ d(T,T') := \inf_{\gamma(0) = T,\gamma(1) = T'} \ell(\gamma)$$ 
is a complete metric, which is realized by a unique geodesic $\gamma$ with $\ell(\gamma) = d(T,T')$ \cite{BHV}. The natural Lebesgue measure for open sets in $\BHV$ is described analogously in Section \ref{stage2} in order to give the volume of small neighborhoods of points in $\BHV_n$; we suspect this might also be of independent interest.

Once properly metrized, tree space can be used to give precise geometric characterizations of collections of phylogenies, and to perform various statistical tests, such as those defined in \cite{SH} and \cite{BLO}. In \cite{BBW}, the matrix of pairwise distances between trees in a set is used as a signature to perform statistical inference. With techniques like this, which operate on the distance matrix instead of the trees themselves, the results are insensitive to isometry; this renders the classification of isometries of $\BHV_n$ extremely relevant. In Theorem \ref{bigthm}, we show that this consists only of permutations of the leaves.

\subsection{Automorphisms versus isometries}
It might seem natural to classify isometries of $\BHV_n$, which is a CAT(0) cube complex (see \cite{MS}), via natural isomorphisms of that structure. However, it is important to note that in general, isometries of cube complexes can exceed their cube complex automorphisms, and if the cubes are endowed with a different metric, an automorphism may not be an isometry at all. As a trivial example, one can consider the integer cubulation of $\mathbb{R}^2$, which in addition to the $D_4\times \mathbb{Z}^2$ lattice isometries, retains the $O(2)\times \mathbb{R}^2$ real isometries, which do not preserve the cube complex structure. This discrepancy was addressed recently in \cite{Bre} - Bregman shows that for a CAT(0) cube complex $C$ with unit euclidean metric on each cube and global metric given by minimal path length, if $Isom(C) \neq Aut(C)$, then there is a full subcomplex $D$ of $C$ admitting a decomposition into a product $E\times \mathbb{R}^n$ , where $E$ is a full subcomplex of $D$. This shows that in some sense, the only additional isometries come from an $\mathbb{R}^n$-type subcomplex, possibly with non-flat curvature. We note that our result gives a counterexample to the converse: the full subcomplex of $\BHV_5$ given by any 5-cycle in the link is $\mathbb{R}^2$ with the singular cone metric Cone($\mathbb{R}^2,5$), but we do not gain any additional isometries.

Besides the proof given in Section \ref{stage1} of this paper, $Aut(\BHV_n)$ is known from the work of Abreu and Pacini classifying cone complex automorphisms of the moduli space $\trop$ of tropical genus 0 curves with $n$ marked points\cite{AP}. Their result is closely related to our Proposition \ref{graphauto}. Inspection of the argument suggests that they are proving the same essential combinatorial fact, through an inductive technique. In fact, our main result could be proved via theirs through a direct application of Lemma \ref{volumebound} to the interior of top-dimensional orthants, analogously to our proof in Section \ref{stage3} that $Aut(L_n) = Isom(L_n)$. We thank Melody Chan for this insightful observation.

\section*{Acknowledgments}
In addition to Melody Chan, who was instrumental in revisions for this second version, we thank Mike Steel and Bernd Sturmfels for help with background, and Bo Lin for insight into the limits of this technique for isometries of tropical tree space (see \cite{LMY}, note that tropical tree space is distinct from the space $\trop$ of tropical curves). The author is especially grateful to her advisor, Andrew Blumberg, for advice and editing. This work was partially supported by National Institute of Health grants 5U54CA193313 and GG010211-R01-HIV, and AFOSR grant FA9550-15-1-0302.

\section{Main Theorem}

\begin{theorem}\label{bigthm}
For $n \geq 3$, the isometry group of $\BHV_n$ is isomorphic to $S_n$. These isometries correspond to permutation of leaf labels.
\end{theorem}

It is clear that a permutation of the leaf labels induces an isometry from $\BHV_n$ to itself, so the following lemmas will build to the converse. This will involve two stages. First, in Section \ref{stage1} we will use the Erd\H{o}s-Ko-Rado theorem to give a new proof that the automorphism group of $L_n$, the spherical simplicial complex of points at distance 1 from the origin, is $S_n$. As we've remarked already, this fact is implied by recent work of \cite{AP}, who computed the automorphisms of $\BHV_n$ as a cone complex. In Section \ref{stage2}, we will then give local bounds on the natural volume measure in $\BHV_n$ to show that any isometry of $\BHV_n$ induces a self-map of the unit sphere $L_n$, and any isometry of the unit sphere to itself is an automorphism of simplicial complexes. Having classified these in the previous section, we conclude in Section \ref{stage3} that any isometric automorphism of $\BHV_n$ must be a relabeling.

\subsection{Link Automorphisms}\label{stage1}
Following \cite{BHV}, $\BHV_n$ can be expressed as a cone on a simplicial complex $L_n$, constructed:
\begin{itemize}
\item A 0-simplex (vertex) $v$ for each $P_v \subset \{1,2,\dots,n\}$ such that $2\leq |P_v| < n/2$. The size $|P_v|$ will often be denoted $k$. Each $P_v$ determines a partition $P_v, P_v^c$ of $[n]$, unique for $k< n/2$. If $n$ is even, we also include a vertex for each pair $P,P^c$ with $|P|=|P^c| = n/2$. 
\item A 1-simplex (edge) $(v,w)$ for each {\it compatible} pair $(P_v,P_v^c)$ and $(P_w,P_w^c)$.  $P_v$ and $P_w$ are said to be compatible if one of the sets $[ P_v\cap P_w, P_v\cap P_w^c, P_v^c \cap P_w, P_v^c \cap P_w^c]$ is empty. We will simplify this condition in Lemma \ref{compat}. 
\item The complex (graph) constructed up to this point is denoted $L_n^1$, the 1-skeleton of $L_n$.
\item $L_n$ is the simplicial complex with a k-simplex, $k>1$, for each $(k+1)$-clique present in $L_n^1$ (i.e. $L_n$ is a {\em flag} simplicial complex).
\item $L_n$ is realized geometrically as a right-angled spherical simplicial complex: for $S^k$ the unit sphere in $\mathbb{R}^k$, each simplex is isometric to 
$$\{(x_1,\dots,x_{k+1})\in S^k : x_i \geq 0 \mbox{ for all }i\} $$
with the spherical metric.
\item Finally, $\BHV_n$ is a right-angled spherical metric cone on $L_n$, as described in \cite{BBI}. Practically, this means that each tree topology is parametrized by $n-3$ non-negative, real coordinates, with the local standard metric in $\mathbb{R}^{n-3}$, as shown in the introduction.
\end{itemize}
 We begin with some facts about $L_n^1$, and then show the automorphism group of $L_n^1$ in Proposition \ref{graphauto}. This gives the automorphisms of $L_n$ via the flag property in Corollary \ref{linkaut}.


\begin{lemma}\label{degreeformula}
The degree of a vertex $v$ of partition size $k$ in $L_n^1$ is 
 given by:
$$ \deg(e_i)= 2^k + 2^{n-k} - n - 4$$
\end{lemma}
\begin{proof}
The degree of $v$ is the number of partitions (of size at least 2) compatible with $P_v,P_v^c$. For $A,A^c$ distinct from $P_v$, we have four compatibility conditions: (1) $A\cap P_v^c = \emptyset$, or equivalently, $A \subset P_v$; (2) $A \cap P_v = \emptyset$, so $A \subset P_v^c$; (3) $A^c \cap P_v = \emptyset$, so $A^c \subset P_v^c$, and (4) $A^c \cap P_v^c = \emptyset$, so $A^c \subset P_v$.

If we have a subset of $[n]$, such that it or its complement satisfies one of these conditions, it can be labeled ($A$ or $A^c$) so that in fact it satisfies (1) or (2). Therefore to count the number of total compatible partitions, we will count subsets $A\subset [n]$ satisfying (1) or (2); that is, nontrivial subsets of sufficient size of $P_v$ or $P_v^c$:
$$\overbrace{\sum_{x=2}^{k-1}{k\choose x}}^{(1)}  + \overbrace{\sum_{x=2}^{n-k-1}{n-k\choose x}}^{(2)}  = (2^k - k -2) + (2^{n-k}-(n-k) -2) = 2^k + 2^{n-k} -n -4.$$

\end{proof}

\begin{lemma}\label{compat} For two distinct partitions $(A,A^c)$, $(B,B^c)$, of size $|A| = k_1$, $|B| = k_2$, $2\leq k_1 \leq k_2 \leq n/2$, $(A,A^c),(B,B^c)$ are compatible iff $A\cap B = \emptyset$ or $A \subset B$. If $k_1 = k_2$, $A \cap B=\emptyset$ is equivalent to compatibility of distinct partitions.
\end{lemma}
\begin{proof}
By the pigeonhole principle, $A^c \cap B^c$ is nonempty. If $B \cap A^c$ is empty, then $B \subseteq A$, which implies by size considerations that $B=A$. For distinct partitions this will not occur. On the other hand, we can have $A \cap B$ or $A \cap B^c$ empty. In the latter case, it is implied that $A \subseteq B$. If $k_1 = k_2< n/2$, then $A \subseteq B$ implies $A = B$.
\end{proof}

\begin{rem} The {\bf Kneser graph} $KG_{n,k}$ is the graph whose vertices correspond to the k-element subsets of a set of $n$ elements, and where two vertices are adjacent if and only if the two corresponding sets are disjoint. Labeling the vertices of $L_n^1$ by the smaller of the two partitions, and sorting by size, it follows immediately that $L_n^1$ contains a unique subgraph $G_k$ isomorphic to $KG_{n,k}$ for each partition size $k = 2,3,\dots, \lceil n/2\rceil - 1$. 
These subgraphs have disjoint vertex sets. If $n$ is even, then there are an additional $\frac{1}{2}{n \choose n/2}$ vertices, pairwise disjoint from each other.
\end{rem}

\begin{prop}\label{graphauto}
The automorphism group $Aut(L_n^1) \cong S_{n}$.
\end{prop}

\begin{proof}
To see that $S_{n}$ is a subgroup of $ Aut(L_n^1)$, we recall that $L_n^1$ is constructed via combinatorial conditions (compatibility) that are independent of choice of label. So any permutation of $\{1,\dots,n\}$ gives an identical graph when constructed with the same notion of compatibility of partitions. Therefore given $\sigma \in S_{n}$, we can map $P=(x_1,x_2,\dots x_k) \mapsto \sigma(P) =(\sigma(x_1),\dots, \sigma(x_k))$, and this preserves adjacency.

It remains then to show that $Aut(L_n^1) \leq S_n$, which we will do by defining an injective group homomorphism $Aut(L_n^1)\rightarrow S_n$. 

Let $\sigma \in Aut(L_n^1)$, and denote by $G_k$ the induced subgraph on the $k$-vertex set $\{v \in V(L_n^1) : |P_v| = k\}$. By Lemma \ref{degreeformula}, the degree of a vertex $v$ is completely and uniquely determined by its size $k$, so $\sigma(v)$ must be contained in $G_k$ as well. Therefore $\sigma$ restricts to a graph automorphism on $G_k$. 

 We now show that this restriction map $Aut(L_n^1) \rightarrow Aut(G_k)$ is injective for $2\leq k < n/2$. Suppose $\sigma|_{G_k} = id$. Then let $N(P_v)^{+1}$ denote the set of neighbors of $v\in G_k\subset L_n^1$ of size $k+1$, i.e.
 $$ N(P_v)^{+1} = \{P_w\in G_{k+1}:  P_v\subset P_w \mbox{ or } P_v \cap P_w = \emptyset\}$$
by Lemma \ref{compat}. Let $P_z = (x_1,x_2,\dots,x_{k+1})\in G_{k+1}$. Then it is straightforward to verify that
$$ P_z = \bigcap_{i=1}^{k+1} N(x_1\dots \hat{x_i}\dots x_{k+1})^{+1}\bigcap_{(y_1,\dots,y_k) \subset P_z^c}N(y_1\dots y_k)^{+1},$$
since $(x_1,x_2,\dots,x_{k+1})$ is the unique partition of size $k+1$ which is compatible with all of its size-$k$ subsets. Then since each $N(P_v)^{+1}$ is preserved for $v \in G_k$, we can conclude that $P_z$ is preserved as well, so $\sigma(G_{k+1}) = G_{k+1}$, which implies that $G_j$ for $j>k$ is preserved under $\sigma$, by repetition of the same argument. Similarly, we have
$$ P_z = \bigcap_{\alpha\in P_z^c}N(x_1\dots x_k, \alpha)^{-1},$$
which shows $\sigma(G_j) = G_j$ for $j<k$ in the same manner.
Since $V(L_n^1) = \bigsqcup_{k=1}^{\lfloor n/2\rfloor} V(G_k)$, we have shown that $\sigma \in \ker(Aut(L_n^1) \rightarrow  Aut(G_k))$ acts trivially on the vertices of $L_n^1$, so must be the trivial automorphism. 

Now following \cite{AGT}, we show that $Aut(G_k) \cong S_n$ for $2\leq k<n/2$. By the Erd\H{o}s-Ko-Rado Theorem, any family of subsets of $\{1,2,\dots,n\}$ of uniform size $k$ having pairwise-nonempty intersection has size $\leq {n-1 \choose k-1}$, and the subsets achieving equality are of the form 
$$G_k^{(i)} = \{v \in G_k: i \in P_v \}$$
for $i \in [n]$.\cite{EKR} Since these partitions pairwise-intersect, they are pairwise disjoint in $G_k$, and by definition form a maximum-size independent set in $G_k$. Correspondingly, $\sigma \in Aut(G_k)$ must induce a permutation on these maximum independent sets, which determines a (surjective) homomorphism $Aut(G_k) \rightarrow S_n$. To see that this is an isomorphism, note that if $\sigma$ fixes the $X(i)$, it must be the identity: suppose $\sigma(v) \neq v$. Then there exists some $j \in P_v$ such that $j \notin P_{\sigma(v)}$. This would imply that $\sigma(G_k^{(j)}) \neq G_k^{(j)}$, a contradiction.

Now we see that $Aut(L_n^1)\cong Aut(G_k)\cong S_n$ (for any/all $2\leq k < n/2$, we really only needed one), which completes the proof.
\end{proof}

\begin{cor}\label{linkaut} The group of simplicial automorphisms of $L_n$ is isomorphic to $Aut(L_n^1)$.
\end{cor}
\begin{proof}
Let $n \geq 3$ be given. First we note that $Aut(L_n) = Aut(L_n^1)$: each simplicial automorphism induces an automorphism of the 1-skeleton, and since $L_n$ contains no simplices with the same 1-skeleton, this map is injective. Then since $L_n$ is a flag complex (\cite{BHV}), given a graph automorphism of $L_n^1$, we can define a canonical extension by sending a $k$-simplex to the $k$-simplex determined by the image of its 1-skeleton $k$-clique.
\end{proof}

\subsection{Measure and Isometry}\label{stage2}

We will now consider the entire metric space $\BHV_n$, and show that the standard embedding of $L_n$ into the unit sphere is invariant under isometry.

There is a natural volume measure $\mu$ on $\mathcal{B}(\BHV_n)$, which is given by the local Lebesgue measure in each orthant. Explicitly, for $A \in \mathcal{B}(\BHV_n)$,
$$ \mu(A) = \sum_S |A\cap S|$$
where $S\cong (\mathbb{R}^+)^{n-3}$ is an orthant of $\BHV_n$ and $|\cdot|$ is the real Lebesgue measure. As we will see in the following lemmas, the volume of small neighborhoods can vary exponentially under translation; this fact is one of the major impediments to statistical techniques in tree space.
\begin{lemma}\label{lebesgue}
For $\sigma\in \Isom(\BHV_n)$, $\sigma$ preserves the volume measure $\mu$ on $\mathcal{\BHV}_n$.
\end{lemma}
\begin{proof}
Let $B_x$ be a ball of radius 1 centered at a point $x \in \BHV^n$. For a fixed orthant $S$, $\sigma$ induces an isometry of $S$ into $\BHV_n$, so $\mu(\sigma (B_x \cap S)) = |B_x \cap S| = \mu(B_x \cap S)$. For a measure zero set $Z$,  $$B_x = \bigsqcup_S B_x\cap \mbox{int}(S) \bigsqcup Z,$$
$$\sigma(B_x) = \sigma\left(\bigsqcup_S B_x\cap \mbox{int}(S) \bigsqcup Z \right)$$
$$\hspace{3em} = \bigsqcup_S \sigma(B_x\cap \mbox{int}(S)) \bigsqcup \sigma(Z),$$
since $\sigma$ is injective. Therefore we conclude that $\mu(\sigma(B_x)) = \bigsqcup_S \mu(B_x \cap S) = \mu(B_x)$.
\end{proof}


\begin{lemma}\label{volumebound} Let $x\in \BHV_n$, with $\{e_1,e_2,\dots,e_p\}$ the set of positive-length edges in $x$, then $0\leq p\leq n-3$. Let $\epsilon>0$ be smaller than the length of $e_i$ for each $i\in \{1,2,\dots,p\}$. Then 
\begin{equation}
\label{upperbound}A_{n-3}(\epsilon) \leq \mu(B_x(\epsilon)) \leq (2n-2p-5)!!\frac{2^p}{2^{n-3}}A_{n-3}(\epsilon),
\end{equation}
where $A_m(\epsilon)$ is the volume of a ball of radius $\epsilon$ in $\mathbb{R}^m$. Furthermore, the lower bound is achieved if and only if $p = n-3$, which means $x$ is binary.
\end{lemma}
\begin{proof}
First, we note that $x$ is contained in a cubical face $F$ of dimension $p$ in $\BHV_n$. Then $F$ is contained in some number $s(F)$ of top-dimensional orthants, each representing a binary tree topology whose partition set contains the partition set of $x$. The restriction on $\epsilon$ ensures that $B_x(\epsilon)$ intersects no lower-dimensional faces, so just as a neighborhood of a point contained in a $p$-face in an $(n-3)$-cube, the restriction of $B_x(\epsilon)$ to each orthant is isometric to $\left(\frac{1}{2^{codim(F)}}\right)$-th of a Euclidean $\epsilon$-ball. So we have that
\begin{equation}\label{ballmeasure}
 \mu(B_x(\epsilon)) = \frac{s(F)}{2^{n-3-p}} A_{n-3}(\epsilon).
\end{equation}
While $s(F)$ is highly dependent on the topology of $F$, we will show that $s(F) \leq (2n-2p-5)!!$, which gives (\ref{upperbound}). 

Instead of describing the topology of $F$ as a list of $p$ internal partitions, we will now consider the internal nodes $y_1,\dots, y_{p+1}$, with degree sequence $d_1,d_2,\dots,d_{p+1}$. Note that 
\begin{equation}\label{degsum}
\sum_{i=1}^{p+1} (d_i-3) = n-p-3,
\end{equation} 
by the fact that the sum of the full degree sequence of a tree is twice the number of edges, so $\sum d_i + n = 2(n+p)$, from which the equality follows. Then
\begin{equation}\label{sF}
s(F) = \prod (2d_i -5)!!
\end{equation}
because locally, each vertex of degree $d_i$ forms the interior node of a star tree with $d_i$ ``leaves" representing the subtrees. So to find a binary tree with the same subtrees as leaves, we count the orthants in $\BHV_{d_i}$, that is, $(2d_i-5)!!$. This choice fixes all other nodes of $F$, so an element of $s(F)$ is specified uniquely by freely choosing a binary tree at each interior node. 

Next we note that $(2d_i-5)!!$ has $d_i-3$ terms greater than 1. For any degree sequence $d_i$, we then have by (\ref{degsum}) that the product (\ref{sF}) has $(n-p-3)$ non-trivial terms, each of which is at least 3, which gives the lower bound. This product is maximized with the degree sequence $n-p,3,3,\dots,3$, for which $s(F) = (2(n-p)-5)!!$, which gives the upper bound. For $p< n-3$, $s(F)$ is strictly greater than $2^{n-3-p}$. For $p=n-3$, we have a coefficient of 1. These two facts show that the lower bound is achieved only for binary trees.
\end{proof}

\begin{cor}\label{volume} Let $n\geq 4$, $c$ the cone point in $\BHV_n$, $x\neq c \in \BHV_n$. Then $\mu(B_c(\epsilon)) > \mu(B_x(\epsilon))$ for $\epsilon < \min_e w_e$, i.e. smaller than the length of the smallest non-zero edge of $x$.
\end{cor}
\begin{proof}
First note that $\mu(B_c(\epsilon)) = \frac{(2n-5)!!}{2^{n-3}}A_{n-3}(\epsilon)$ for any $\epsilon>0$, where $A_m$ is the volume of the unit ball in $\mathbb{R}^m$. Then for $x \neq c$, $p\geq 1$, so by Lemma \ref{volumebound},
$$ \mu(B_x(\epsilon)) \leq (2n-7)!!\frac{2}{2^{n-3}}A_{n-3}(\epsilon).$$
But since $2 < 2n-5$, $\mu(B_x(\epsilon)) < \mu(B_c(\epsilon))$.
\end{proof}

\subsection{Proof of Main Theorem}\label{stage3}

\begin{proof}
Let $n \geq 4$ be given.

Each of the relabeling automorphisms of $L_n$ is an isometry, and it extends in the obvious way to an isometry of $\BHV^n$ by relabeling the leaves of an arbitrary tree, so we can conclude that $S_{n} \leq \Isom(\BHV_n)$.\footnote{Equivalently, an automorphism of a cube complex with uniform euclidean metric is automatically an isometry.}
Conversely, it remains to be shown that $\Isom(\BHV^n) \leq S_n$. Let $\sigma \in \Isom(\BHV^n)$ be given. 
\begin{enumerate}
\item Let $B_x(\epsilon)$ denote the set of points at distance at most $\epsilon$ from $x$. Then by definition of an isometry, $\sigma( B_x(\epsilon)) = B_{\sigma(x)}(\epsilon)$ for all $\epsilon$.
\item For $x \neq c$, $ \epsilon < \min_e w_e$, the measure $\mu(B_x(\epsilon)) < \mu(B_c(\epsilon))$ by Cor. \ref{volume}.  
\item We conclude that $\sigma(c) = c$ by Lemma \ref{lebesgue}, so $\sigma(B_c(1)) = B_c(1)$.
\item Since $L_n = \partial(\overline{B_c(1)})$ is the set of points at distance $1$ from $c$, we conclude that $\sigma(L_n) = L_n$. 
\item In the remainder of the proof, we will show that $Isom(L_n) = Aut(L_n) \cong S_n$, and this will give the titular result.
\end{enumerate}
Let $\sigma \in Isom(L_n)$be given.  Let $x \in L_n$ be a binary tree, so $x$ is contained in the interior of an ($n-4$)-simplex. Then by Lemma \ref{volumebound} and Lemma \ref{lebesgue}, $\sigma(x)$ is also necessarily a binary tree, and so contained in the interior of an $(n-4)$-simplex in $L_n$. An isometry which restricts to $\tau: \mbox{int}(\Delta^{n-4})\to \mbox{int}(\Delta^{n-4})$ on the interior of an $(n-4)$-simplex must extend by continuity to an isometry $\bar{\tau}:\Delta^{n-4}\to \Delta^{n-4}$. Such an isometry is a simplicial map, sending $k$-simplices to $k$-simplices. But every $k$-simplex in $L_n$ is on the boundary of a maximal simplex (equivalently, every non-binary tree has a choice of additional edges making it binary), so we conclude that $\sigma$ is a simplicial map from $L_n$ to $L_n$, i.e. $\sigma \in Aut(L_n)$. Since every automorphism is an isometry, we conclude $Isom(L_n) \cong Aut(L_n)$, and by Corollary \ref{linkaut}, $Aut(L_n^1) \cong Aut(L_n) \cong S_n$. 
\end{proof}

\end{document}